\documentclass{amsart}

\usepackage[english]{babel}

\usepackage{amsmath, amsfonts, amssymb, amsthm, faktor}

\usepackage[all]{xy}
\usepackage{tikz}
\usetikzlibrary{math}
\usetikzlibrary{patterns}
\usepackage{caption}
\usepackage{subfig}

\usepackage{caption}
\usetikzlibrary{arrows,chains,matrix,positioning,scopes,decorations.pathreplacing,
decorations.pathmorphing,decorations.markings,arrows.meta}

\usepackage{aliascnt}
\usepackage[colorlinks, linktocpage, allcolors=black,breaklinks]{hyperref}

\usepackage{enumerate}

\usepackage{verbatim}

\theoremstyle{definition}
\newtheorem{theorem}{Theorem}[section]

\newtheorem{proposition}[theorem]{Proposition}
\newtheorem{corollary}[theorem]{Corollary}
\newtheorem{lemma}[theorem]{Lemma}

\newcommand{\G}{\Gamma}
\newcommand{\Z}{\mathbb Z}
\newcommand{\acts}{\curvearrowright}
\newcommand{\K}{\kappa}

\begin{document}
\title[Continuous Edge Chromatic Numbers]{Continuous Edge Chromatic Numbers of Abelian Group Actions}

\author{Su Gao}
\address{School of Mathematical Sciences and LPMC, Nankai University, Tianjin 300071, P.R. China}
\email{sgao@nankai.edu.cn}
\thanks{The first author acknowledges the partial support of his research by the National Natural Science Foundation of China (NSFC) grants 12250710128 and 12271263.}

\author{Ruijun Wang}
\address{School of Mathematical Sciences and LPMC, Nankai University, Tianjin 300071, P.R. China}
\email{rwang@mail.nankai.edu.cn}

\author{Tianhao Wang}
\address{School of Mathematical Sciences and LPMC, Nankai University, Tianjin 300071, P.R. China}
\email{Tianhao\_Wang@qq.com}



\date{}

\maketitle

\begin{abstract}
We prove that for any generating set $S$ of $\mathbb {Z}^n$, the continuous edge chromatic number $\chi'_c(G)$ of the Schreier graph of the Bernoulli shift action $G=F(S,2^{\mathbb{Z}^n})$ is $\chi'(G)+1=|S|+1$. In particular, for the standard generating set, the continuous edge chromatic number of $F(2^{\mathbb {Z}^n})$ is $2n+1$. 
\end{abstract}

\section{Introduction}
In the seminal paper \cite{KST99}, Kechris, Solecki and Todorcevic initiated the study of descriptive combinatorics of locally finite Borel graphs. In particular, they considered general bounds for Borel chromatic numbers of locally finite Borel graphs. The Borel edge chromatic number was also considered, primarily as the Borel chromatic number of the dual graph, and they showed that any locally finite Borel graph $G$ has countable Borel edge chromatic number $\chi'_B(G)$, and if $G$ is of bounded degree $\leq k$, then $\chi'_B(G)\leq 2k-1$. Subsequently, researchers computed the Borel chromatic numbers and Borel edge chromatic numbers of various Schreier graphs of countable group actions. For example, in \cite{Marks}, Marks studied combinatorics of free products of two marked groups, and showed that the Borel chromatic number $\chi_B(F(2^{\mathbb F_n}))=2n+1$, where $F(2^{\mathbb{F}_n})$ denotes the Schreier graph on the free part of the Bernoulli shift action of $\mathbb{F}_n$, the free group with $n$ generators, on $2^{\mathbb{F}_n}$. For an overview of the entire field of descriptive combinatorics we refer the reader to the survey \cite{KM} by Kechris and Marks.

The continuous chromatic number was first studied by the first author of the present paper and Jackson. In \cite{GJ15}, it was shown that there is a continuous proper $4$-coloring of $F(2^{\mathbb Z^n})$ for each $n>1$. Later in \cite{GJKS23}, it was shown that there is no continuous proper $3$-coloring of $F(2^{\mathbb Z^n})$ for each $n>1$. Thus the continuous chromatic number $\chi_c(F(2^{\mathbb Z^n}))=4$ for each $n>1$. In contrast, it was shown in \cite{CJMST} and \cite{GJKS24} by different methods that the Borel chromatic number $\chi_B(F(2^{\mathbb Z^n}))=3$ for each $n>1$. This was the first time that the Borel and continuous chromatic numbers of a locally finite Borel graph were observed to be different. 

A similar phenomenon occurs when considering the Borel and continuous edge chromatic numbers of $F(2^{\mathbb{Z}^n})$. The Borel edge chromatic numbers were studied independently in \cite{BHT} (for $n=2$), \cite{CU}, \cite{GR23} and \cite{Weilacher}. The conclusion is that for the Schreier graph of $F(2^{\mathbb{Z}^n})$ with the standard generating set, the Borel edge chromatic number $\chi'_B(F(2^{\mathbb{Z}^n}))$ equals the usual edge chromatic number $\chi'(F(2^{\mathbb{Z}^n}))$, which is $2n$. It is worth noting that Weilacher \cite{Weilacher} showed that for any symmetric generating set $S$ of $\mathbb{Z}^n$ where $S$ does not contain the identity, letting $F(S, 2^{\mathbb{Z}^n})$ be the corresponding Schreier graph, we have $\chi'_B(F(S,2^{\mathbb{Z}^n}))=\chi'(F(S,2^{\mathbb{Z}^n}))=|S|$. Thus the computation of the Borel edge chromatic numbers of $F(2^{\mathbb{Z}^n})$ was complete. As for the continuous edge chromatic number, we only had, from \cite{GJKS23}, the result for $n=2$ that $\chi'_c(F(2^{\mathbb Z^2}))=5$ for the standard generating set. This already showed that the Borel and continuous edge chromatic numbers could differ, but left the computation of a lot of cases open. 

In this paper, we complete this line of research by showing that, for any symmetric generating set $S$ of $\mathbb{Z}^n$ not containing the identity, the continuous edge chromatic number $\chi'_c(F(S,2^{\mathbb{Z}^n}))=\chi'(F(S,2^{\mathbb Z^n}))+1=|S|+1$. In particular, we have $\chi'_c(F(2^{\mathbb{Z}^n}))=2n+1$. This answers Question 10.3 of \cite{GS}.

\begin{theorem}\label{thm:1.1}
    For any integer $n\geq 1$ and any symmetric generating set $S$ of $\mathbb {Z}^n$ not containing the identity, let $G=F(S,2^{\mathbb{Z}^n})$. Then $\chi'_c(G)=\chi'(G)+1=|S|+1$.
\end{theorem}

The rest of the paper is organized as follows. In Section~\ref{sec:2} we fix the notation to be used throughout the paper. In Section~\ref{sec:3} we prove the lower bound for the continuous edge chromatic number. Then in Section~\ref{sec:4} we construct a proper edge coloring to witness Theorem~\ref{thm:1.1},  first by assuming $S$ to be the standard generating set and then generalizing to the general case.

{\em Acknowledgment.} We thank the anonymous referees for useful comments and suggestions which improved the paper.

\section{Preliminaries\label{sec:2}}
In this section we define our basic notions and fix notation. We use standard concepts and terminology from graph theory, which can be found in \cite{Bo}.

As usual, a {\em graph} $G$ is a pair $(V(G),E(G))$, where $V(G)$ is a set and $E(G)$ is a set of unordered pairs of distinct elements of $V(G)$. Thus our graphs are {\em undirected}. Here $V(G)$ is the set of {\em vertices} of $G$, and $E(G)$ is the set of {\em edges} of $G$. Throughout the paper, when there is no danger of confusion, the vertex operator $V$ of $V(G)$ is sometimes omitted. If $x, y\in V(G)$ are distinct and $e=\{x,y\}\in E(G)$, we also write $e=(x,y)=(y,x)$ and say that $x$ and $y$ are {\em incident} with the edge $e$. For any $x\in V(G)$, the {\em degree} of $x$ is defined as the number of distinct edges incident with $x$. A graph $G$ is {\em locally finite} if for all $x\in V(G)$, the degree of $x$ is finite. If $x,y\in V(G)$ are distinct, then a {\em path} $p$ from $x$ to $y$ is a sequence $x_0, x_1,\dots, x_k$ of distinct vertices of $G$ such that $x_0=x$, $x_k=y$, and for all $0\leq i<k$, $(x_i,x_{i+1})\in E(G)$; here $k$ is the {\em length} of the path $p$.

For a graph $G$ and a set $\K$ of colors, a {\em proper edge $\K$-coloring} of $G$ is a map $c: E(G)\to \K$ such that $c((x,y))\neq c((z,w))$ if the distinct edges $(x,y),(z,w)$ have a vertex in common. In the sequel, we write $c(x,y)$ for $c((x,y))$ for simplicity. The {\em edge chromatic number} of $G$, denoted by $\chi'(G)$, is the least cardinality of a set $\K$ such that there exists a proper edge $\K$-coloring for $G$. When the graph $G$ is a topological graph, i.e., when $V(G)$ is a topological space, we may consider {\em continuous} (or {\em Borel}) proper edge colorings of $G$, and define its {\em continuous edge chromatic number}, denoted as $\chi'_c(G)$, and its {\em Borel edge chromatic number}, denoted as $\chi'_B(G)$. Similarly, we define proper (vertex) colorings and the continuous and Borel chromatic numbers, denoted by $\chi_c(G)$ and $\chi_B(G)$ respectively.

On any locally finite graph $G$, we define the {\em path distance} $\rho$ on $G$ as follows. If $x,y\in V(G)$ are in the same connected component (i.e. there is a path from $x$ to $y$), then let $\rho(x,y)$ be the length of a shortest path from $x$ to $y$; otherwise let $\rho(x,y)=\infty$. Recall that a locally finite graph $G$ is {\em regular} if the degrees of all vertices of $G$ are the same.

An important class of regular topological graphs consists of the graphs induced by actions of finitely generated groups. A {\em marked group} is a pair $(\Gamma, S)$, where $\Gamma$ is a group and $S$ is a finite generating set of $\Gamma$. Usually we also require $S=S^{-1}$ and $1_\Gamma\not\in S$. In this paper this is assumed tacitly. When the generating set is standard or otherwise understood, we omit specifying the set $S$ and say that $\Gamma$ is a marked group. For example, for any $n\geq 1$, $\mathbb{Z}^n$ has a standard generating set $\{\pm e_i\colon 1\leq i\leq n\}$ where each $e_i$ has the $i$-th coordinate $1$ and $j$-th coordinate $0$ for $j\neq i$. The {\em Cayley graph} $C(\Gamma, S)$ of a marked group $(\Gamma, S)$ is defined by
$$ V(C(\Gamma, S))=\Gamma $$
and
$$ E(C(\Gamma, S))=\{(g,h)\in \Gamma^2\,:\, \exists s\in S\ h=sg\}. $$
When there is an action of a marked group $\Gamma$ on a set $X$, the {\em Schreier graph} of the action $\Gamma\acts X$ on $X$, denoted $G(\Gamma, S,X)$, is defined by
$$ V(G(\Gamma, S,X))=X $$
and
$$ E(G(\Gamma, S,X))=\{(x,y)\in X^2\,:\, \exists s\in S\ y=s\cdot x\}. $$

The Schreier graph will be particularly nice when the action is free; in this case the Schreier graph on each orbit of the action will be a copy of the Cayley graph $C(\Gamma,S)$.  Let $F(\Gamma, X)=\{ x \in X\colon  \forall g \neq 1_\G \ (g\cdot x \neq x)\}$ be the {\em free part} of the action $\Gamma\curvearrowright X$. $F(\Gamma, X)$ is an invariant set and the induced action of $\Gamma$ on $F(\Gamma, X)$ is free.  We denote by $F(\Gamma,S,X)$ the Schreier graph on $F(\Gamma, X)$. We write $F(X)$ for $F(\Gamma, X)$ and $F(S, X)$ for $F(\Gamma, S, X)$ if the group $\Gamma$ and its action on $X$ are understood. Furthermore, when the generating set $S$ is standard or otherwise understood, we will abuse notation and simply write $F(X)$ for the Schreier graph $F(S, X)$.

In this paper, we will work with the setup where $X$ is a Polish space, i.e., a separable and completely metrizable space, $\Gamma$ is a finitely generated group with the discrete topology, and the action $\Gamma\curvearrowright X$ is continuous. 

The Bernoulli shift action induces a natural Schreier graph. Let $A$ be a finite set with the discrete topology. Assume $|A|\geq 2$. For a countable discrete group $\G$, the space $A^\G$ is equipped with the usual product topology and is homeomorphic to the Cantor space $2^\mathbb N=\{0,1\}^{\mathbb{N}}$, which is compact Polish.

The {\em Bernoulli shift action} $\G\curvearrowright A^\G$ is defined by
\[
(g \cdot x)(h)= x(hg)
\]
for $g,h \in \G$ and $x \in A^\G$. 
This action is continuous. The free part $F(\Gamma, A^\Gamma)$ is an invariant dense $G_\delta$ subset of $A^\Gamma$, hence is a Polish space. 

We are now ready to compute the continuous edge chromatic number $\chi'_c(F(S,2^{\mathbb{Z}^n}))$ for a generating set $S$ of $\mathbb{Z}^n$.

\section{The Lower Bound\label{sec:3}}

In this section we prove that for any generating set $S$ of $\mathbb{Z}^n$, $\chi'_c(F(S, 2^{\mathbb{Z}^n}))>\chi'(F(S,2^{\mathbb{Z}^n}))$. We will need some results of \cite{GJKS23}. To state the results we need to review some related concepts.

Let $(\G, S)$ be a marked group. We will consider the Bernoulli shift action of $\G$ on $A^\G$ for some finite set $A$. A {\em $\G$-subshift} $Y$ is a closed invariant subset of $A^{\G}$. A {\em $\G$-pattern} is a map $p: F\to A$ for some finite $F\subseteq\G$. If $p$ is a $\G$-pattern and $x\in A^\G$, we say that $p$ {\em occurs} in $x$ if $p\subseteq g\cdot x$ for some $g\in \G$. Note that this is an invariant notion, that is, if $p$ occurs in $x$ then $p$ occurs in any $h\cdot x$ for $h\in\G$. A {\em $\G$-subshift of finite type} is a $\G$-subshift $Y\subseteq A^\G$ for which there is a finite set $\{p_1,\dots, p_k\}$ of $\G$-patterns such that for any $x\in A^\G$, $x\in Y$ if and only if none of the patterns $p_1,\dots, p_k$ occur in $x$. In this case we say that $Y$ is {\em described} by $(A; p_1,\dots, p_k)$.

A graph $H$ is a {\em $(\G,S)$-graph} if there is an action of $\G$ on $V(H)$ such that $H=G(\G, S, V(H))$. The Schreier graphs $G(\G,S,X)$ and $F(\G,S,X)$ are examples of $(\G,S)$-graphs, but we will also consider finite $(\G,S)$-graphs. When the generating set $S$ is standard, $(\G,S)$-graphs are also simply called {\em $\G$-graphs}. 

We will work with the following example of a finite $\mathbb{Z}^n$-graph. Given positive integers $q_1,\dots, q_n$, the vertex set of the graph $T_{q_1,\dots, q_n}$ is the quotient of $\Z^n$ by the
relation $(a_1,\dots,a_n)\sim (b_1,\dots,b_n)$ if and only if $a_1 \equiv b_1\!\mod q_1$, $\dots,a_n \equiv b_n\!\mod q_n$. The group $\mathbb{Z}^n$ acts naturally on $T_{q_1,\dots, q_n}$, and this action induces a $\mathbb{Z}^n$-graph on $T_{q_1,\dots, q_n}$ with the standard generating set of $\mathbb{Z}^n$. Note that $T_{q_1,\dots, q_n}$ has cardinality $q_1\cdots q_n$.

If $Y$ is a $\mathbb{Z}^n$-subshift of finite type described by $(A; p_1,\dots p_k)$, then we say that a map
$\varphi \colon T_{q_1,\dots,q_n}\to A$ {\em respects} $Y$ if none of $p_1,\dots, p_k$ occur in $\varphi\circ \pi\in A^{\mathbb{Z}^n}$, where $\pi \colon \Z^n \to T_{q_1,\dots,q_n}$ is the quotient map.

If a group $\G$ acts on both $X$ and $Y$, we say a map $\theta\colon X\to Y$ is
{\em equivariant} if $\theta(g\cdot x)=g \cdot \theta(x)$ for all $ g\in \G$ and $x \in X$.

\begin{theorem}\label{thm:weakneg}(\cite[Theorem 2.7.1]{GJKS23})
Let $A$ be a finite set and let $Y\subseteq A^{\Z^n}$ be a $\mathbb{Z}^n$-subshift of finite type. If there is a continuous equivariant map
$\theta\colon F(2^{\Z^n})\to Y$, then for all $q_1,\dots,q_n\geq 2$ and all sufficiently large $k$ there is
$\varphi\colon T_{q^k_1,\dots,q^k_n}\to A$ which respects $Y$.
\end{theorem}

Recall that a {\em  matching} of a graph $G$ is a subset $M\subseteq E(G)$ such that every vertex of $G$ is incident with at most one of the edges of $M$. A matching $M$ of $G$ is a {\em perfect matching} if every vertex of $G$ is incident with exactly one edge in $M$. We may view a perfect matching of a graph $G$ as a function $\mu: V(G)\to V(G)$ with the property that for any $x\in V(G)$, $(x,\mu(x))\in E(G)$ and $\mu^2(x)=x$. In this sense we may speak of continuous perfect matchings for topological graphs $G$. The following theorem is a generalization of \cite[Theorem 3.1.3]{GJKS23} with a similar proof.

\begin{theorem}\label{thm:ncpm} Let $S$ be a generating set of $\mathbb{Z}^n$. Then there is no continuous perfect matching of $F(S,2^{\mathbb Z^n})$.
\end{theorem}

\begin{proof}
Suppose $S=\{\pm s_1,\pm s_2,\cdots,\pm s_m\}$, where none of the $s_i$ are the identity 
$\bar{0}$. Rewrite $S=\{u_1,\dots, u_{2m}\}$. Assume a continuous perfect matching $\mu$ of $F(S,2^{\mathbb Z^n})$ exists. Note that the vertex set of $F(S,2^{\mathbb{Z}^n})$ is the same as that of $F(2^{\mathbb{Z}^n})$. Define 
$\eta:F(2^{\mathbb Z^n})\rightarrow S$ by letting $\eta(x)$ to be the unique element $g\in S$ such that $\mu(x)=g\cdot x$. Then $\eta$ is continuous. 

Now define 
$$\theta:F(2^{\mathbb Z^n})\rightarrow S^{\mathbb{Z}^n}$$
by $\theta(x)(g)=\eta(g\cdot x)$ for $x\in F(2^{\mathbb{Z}^n})$ and $g\in \mathbb{Z}^n$. Then $\theta$ is a continuous equivariant map.

Let $Y\subseteq S^{\mathbb{Z}^n}$ be the $\mathbb{Z}^n$-subshift of finite type described by $2m(2m-1)$ many $\mathbb{Z}^n$-patterns $\{p_{i,j}\colon 1\leq i, j\leq 2m, i\neq j\}$, where
$p_{i,j}:\{\bar{0},u_i\}\rightarrow S$ is defined by $p_{i,j}(\bar{0})=u_i$ and $p_{i,j}(u_i)=-u_j$. Then we in fact have $\theta: F(2^{\mathbb Z^n})\rightarrow Y$.

Let $q_1,q_2,\cdots,q_n$ be odd. By Theorem~\ref{thm:weakneg}, for some large enough $k$ there exists $\varphi\colon T_{q_1^k,q_2^k,\cdots,q_n^k}\rightarrow S$ which respects $Y$. Consider the $(\mathbb{Z}^n, S)$-graph with vertex set $T_{q_1^k,q_2^k,\cdots,q_n^k}$, and denote it as $T$. Here since the group $\mathbb{Z}^n$ naturally acts on $T_{q_1^k, q_2^k,\dots, q_n^k}$, the action of each element $s\in S$ on an element $x\in T_{q_1^k,q_2^k,\dots, q_n^k}$ is well defined and hence gives rise to the Schreier graph $T$. As $\varphi$ respects $Y$, $\varphi$ induces a perfect matching of $T$. In fact, if $x\in T$ and $\varphi(x)=s$, then we must have $\varphi(s\cdot x)=-s$ by the definition of $Y$, and hence the edge set $\{(x,s\cdot x)\colon x\in T, s=\varphi(x)\}$ is a perfect matching of $T$. This is a contradiction as the cardinality of $T$ is odd.
\end{proof}

\begin{corollary}\label{cor:lb}
    For any generating set $S$ of $\mathbb {Z}^n$, letting $G=F(S,2^{\mathbb{Z}^n})$, then $\chi'_c(G)>\chi'(G)$.
\end{corollary}
\begin{proof} It is easy to see that $\chi'(G)=\chi'(C(\mathbb{Z}^n, S))=|S|$. 
    Assume there is a continuous proper edge $|S|$-coloring of $F(S,2^{\mathbb Z^n})$. Then every vertex of $F(S,2^{\mathbb{Z}^n})$ has exactly one edge of each color incident with it. Thus, for any color $u\in S$, the set of edges with color $u$ is a continuous perfect matching of $F(S,2^{\mathbb Z^n})$, contradicting Theorem~\ref{thm:ncpm}.
\end{proof}

\section{Continuous Proper Edge Colorings of $F(2^{\Z^n})$\label{sec:4}}
In this section we construct a continuous proper edge $(|S|+1)$-coloring of $F(S, 2^{\mathbb{Z}^n})$ for any $n\geq 1$ and generating set $S$ of $\mathbb{Z}^n$. 

Note that any generating set $S$ of $\mathbb{Z}^n$ has at least $2n$ many elements. Thus throughout the rest of this section we fix once and for all $2n+1$ many colors to be used in the construction of our edge coloring:
$$ c_1,\dots, c_n, 1, \dots, n+1. $$
We will name more colors as needed.

\subsection{Proper edge colorings of a rectangle}

We first work with $n$-dimensional rectangles in $\mathbb{Z}^n$. 

We fix some more terminology about an $n$-dimensional rectangle $R$ in $\mathbb{Z}^n$. We say that $R$ is of {\em size} $a_1\times\cdots\times a_n$ if there is $(b_1,\dots, b_n)\in\mathbb{Z}^n$ such that
$$ R=[b_1, b_1+a_1]\times \cdots \times [b_n,b_n+a_n]. $$
If $R$ is of the above form, then a vertex $x$ in $R$ is said to have {\em coordinates} $(x_1,\cdots,x_n)$ in $R$ if $x=(b_1+x_1,\dots, b_1+x_n)$. We view $R$ as an induced subgraph of the Cayley graph $C(\mathbb{Z}^n)$ of $\mathbb{Z}^n$ with the standard generating set $\{\pm e_i\colon 1\leq i\leq n\}$. We say an edge $(x,y)$ in $C(\mathbb{Z}^n)$ is {\em parallel} to $e_i$ if $e_i+x=y$ or $e_i+y=x$. Let $\partial R$ denote the set of all edges in $R$ which are adjacent to at least one edge in $C(\mathbb{Z}^n)$ that is not in $R$. We call $\partial R$ the {\em boundary} of $R$.  
An edge in $C(\mathbb{Z}^n)$ is said to be {\em adjacent to} $R$ if it is not in $R$ but is adjacent to some edge of $R$. An edge coloring  of $C(\mathbb{Z}^n)$ is said to satisfy the {\em boundary condition} for $R$ if 
\begin{enumerate}
    \item it is a proper edge coloring of $R$ and its adjacent edges;
    \item for all $1\leq i\leq n$, all edges adjacent to $R$ and parallel to $e_i$ have the same color $c_i$.
\end{enumerate}

\begin{figure}[h]
        \centering
        \begin{tikzpicture}[scale=1.5]
            \draw (-2,1) to (-1.73,1);
            \draw (-1.27,1) to (-0.71,1);
            \draw (-0.29,1) to (0.29,1);
            \draw (2,1) to (1.73,1);
            \draw (1.27,1) to (0.71,1);
           
 \draw (-2,0) to (-1.73,0);
            \draw (-1.27,0) to (-0.71,0);
            \draw (-0.29,0) to (0.29,0);
            \draw (2,0) to (1.73,0);
            \draw (1.27,0) to (0.71,0);

          \draw (-2,-1) to (-1.73,-1);
            \draw (-1.27,-1) to (-0.71,-1);
            \draw (-0.29,-1) to (0.29,-1);
            \draw (2,-1) to (1.73,-1);
            \draw (1.27,-1) to (0.71,-1);
            
\draw (1,-2) to (1,-1.73);
            \draw (1,-1.27) to (1,-0.71);
            \draw (1,-0.29) to (1,0.29);
            \draw (1,2) to (1,1.73);
            \draw (1,1.27) to (1,0.71);

\draw (0,-2) to (0,-1.73);
            \draw (0,-1.27) to (0,-0.73);
            \draw (0,-0.27) to (0,0.29);
            \draw (0,2) to (0,1.73);
            \draw (0,1.27) to (0,0.71);

\draw (-1,-2) to (-1,-1.73);
            \draw (-1,-1.27) to (-1,-0.71);
            \draw (-1,-0.29) to (-1,0.29);
            \draw (-1,2) to (-1,1.73);
            \draw (-1,1.27) to (-1,0.71);

            \node[circle, 
 radius=0.25,draw=lightgray] (1) at (1.5,1) {$c_1$};
            \node[circle,
radius=0.25 ,draw=lightgray] (1) at (1.5,-0) {$c_1$};
            \node[circle,
minimum width =0.5 ,
minimum height =0.5 ,draw=lightgray] (1) at (1.5,-1) {$c_1$};
            \node[circle,
minimum width =0.5 ,
minimum height =0.5 ,draw=lightgray] (1) at (-1.5,1) {$c_1$};
            \node[circle,
minimum width =0.5 ,
minimum height =0.5 ,draw=lightgray] (1) at (-1.5,-0) {$c_1$};
            \node[circle,
minimum width =0.5 ,
minimum height =0.5 ,draw=lightgray] (1) at (-1.5,-1) {$c_1$};
            \node[circle,
radius=0.25,draw=gray] (1) at (1,0.5) {3};
            \node[circle,
minimum width =0.5 ,
minimum height =0.5 ,draw=lightgray] (1) at (1,1.5) {$c_2$};
            \node[circle,
minimum width =0.5 ,
minimum height =0.5 ,draw=lightgray] (1) at (1,-1.5) {$c_2$};
            \node[circle,
minimum width =0.5 ,
minimum height =0.5 ,draw=lightgray] (1) at (0.5,1) {1};
            \node[circle,
minimum width =0.5 ,
minimum height =0.5 ,draw=lightgray] (1) at (-0.5,1) {2};
            \node[circle,
minimum width =0.5 ,
minimum height =0.5 ,draw=lightgray] (1) at (-1,0.5) {3};
            \node[circle,
minimum width =0.5 ,
minimum height =0.5 ,draw=lightgray] (1) at (-1,1.5) {$c_2$};
            \node[circle,
minimum width =0.5 ,
minimum height =0.5 ,draw=lightgray] (1) at (-1,-1.5) {$c_2$};
            \node[circle,
minimum width =0.5 ,
minimum height =0.5 ,draw=lightgray] (1) at (-1,-0.5) {1};
            \node[circle,
minimum width =0.5 ,
minimum height =0.5 ,draw=lightgray] (1) at (-0.5,-1) {2};
            \node[circle,
minimum width =0.5 ,
minimum height =0.5 ,draw=lightgray] (1) at (0.5,-1) {1};
            \node[circle,
minimum width =0.5 ,
minimum height =0.5 ,draw=lightgray] (1) at (1,-0.5) {2};
            \node[circle,
minimum width =0.5 ,
minimum height =0.5 ,draw=lightgray] (1) at (-0.5,-0) {2};
            \node[circle,
minimum width =0.5 ,
minimum height =0.5 ,draw=lightgray] (1) at (0.5,-0) {1};
            \node[circle,
radius=0.25,draw=gray] (1) at (-0,-0.5) {$c_1$};
            \node[circle,
minimum width =0.5 ,
minimum height =0.5 ,draw=lightgray] (1) at (-0,0.5) {3};
            \node[circle,
minimum width =0.5 ,
minimum height =0.5 ,draw=lightgray] (1) at (-0,1.5) {$c_2$};
            \node[circle,
minimum width =0.5 ,
minimum height =0.5 ,draw=lightgray] (1) at (-0,-1.5) {$c_2$};
        \end{tikzpicture}
        \caption{An edge coloring satisfying the boundary condition.}
        \label{example}
    \end{figure}
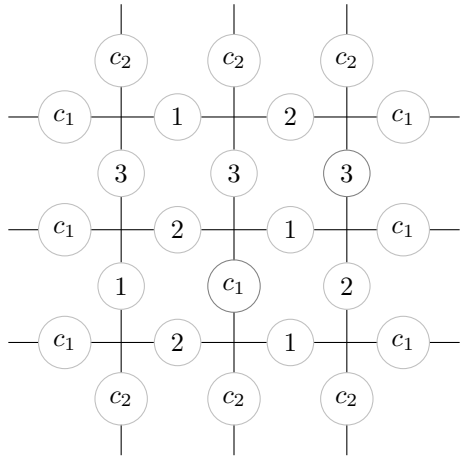

Figure \ref{example} illustrates an example of an edge coloring of a $2$-dimensional rectangle of size $2\times 2$ satisfying the boundary condition.

\begin{lemma}\label{lem:bc1}
    For any $n\geq 1$ and any $n$-dimensional rectangle $R$, there is a proper edge $(2n+1)$-coloring satisfying the boundary condition for $R$.
\end{lemma}
\begin{proof}
    By induction on $n$. 
    For $n=1$ the lemma is obvious. For the inductive case $n>1$, consider an $n$-dimensional rectangle $R$ of size $a_1\times\cdots\times a_n$. We only need to define the edge coloring for edges either in $R$ or adjacent to $R$.

    We first divide $R$ into $a_n+1$ many layers as illustrated in Figure \ref{layer}. 
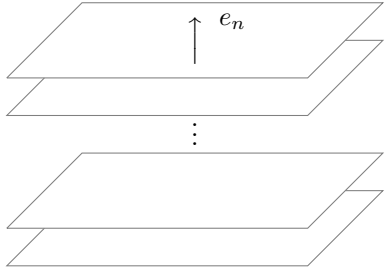
\begin{figure}[h]
        \centering
        \begin{tikzpicture}[yscale=0.5]
            \draw[draw=gray] (0.5,1) to (0,0) to (4,0) to (5,2) to (4.5,2);
            \draw[draw=gray] (0,1) to (4,1) to (5,3) to (1,3) to (0,1);
            \draw[draw=gray] (0.5,5) to (0,4) to (4,4) to (5,6) to (4.5,6);
            \draw[draw=gray] (0,5) to (4,5) to (5,7) to (1,7) to (0,5);
            \node at (2.5,3.25) {.};
            \node at (2.5,3.5) {.};
            \node at (2.5,3.75) {.};
            \node at (2.5,6) {$\Big\uparrow$};
            \node at (3,6.5) {$e_n$};
            
        \end{tikzpicture}
        \caption{Layers as subgraphs of $R$.}
        \label{layer}
    \end{figure}
By the inductive hypothesis, there is a proper edge $(2n-1)$-coloring of each layer of $R$ satisfying the boundary condition. We fix such a proper edge $(2n-1)$-coloring and use it to color all layers of $R$ identically. Note that the colors used so far are among $c_1,\dots, c_{n-1}, 1, \dots, n$, and the remaining uncolored edges are those parallel to $e_n$.
    
    Next, consider an arbitrary edge $(x,y)$ parallel to $e_n$. If $(x,y)$ is adjacent to $R$, then we color it by $c_n$, as required by the boundary condition. Otherwise, assume $(x,y)$ is in $R$. Then $x,y\in R$. Let $C_x$ be the set of colors among $c_1,\dots, c_{n-1}, 1, \dots, n$ which are already used to color an edge incident with $x$. Similarly define $C_y$. Then $C_x=C_y$ and both have cardinality $2n-2$. Thus there exists a color $c^*\in \{c_1,\dots, c_{n-1}, 1, \dots, n\}$ such that $c^*\not\in C_x=C_y$. 

Now the uncolored edges can be decomposed into $(a_1+1)\cdots (a_{n-1}+1)$ many paths parallel to $e_n$, i.e., each path consists of only edges parallel to $e_n$, and so each path is of the form $x, x+e_n, x+2e_n,\dots, x+ke_n$ for some $x$ and postive integer $k$. For each such path $p$ there is a color $c^*\in \{c_1,\dots, c_{n-1}, 1, \dots, n\}$ such that $c^*\not\in C_x$ for all vertices $x$ in $p$. We can then use the two colors $c^*$ and $n+1$ in alternation to color all the edges of $p$, so that the resulting edge coloring is proper.
\end{proof}

\begin{lemma}\label{lem:bc2}
    Let $R$ be an $n$-dimensional rectangle of size $a_1\times\cdots\times a_n$. If $a_i$ is odd for some $1\leq i\leq n$, then there is a proper edge $2n$-coloring satisfying the boundary condition for $R$.
\end{lemma}
\begin{proof} We will use only the colors $c_1,\dots, c_n, 1,\dots, n$. If $n=1$ it is easy to see that we can use $c_1$ to color the two edges adjacent to $R$ and use $c_1$ and $1$ in alternation to color the edges of $R$; the resulting edge coloring is proper and satisfies the boundary condition.

Suppose $n>1$ and, without loss of generality, assume $a_n$ is odd. Similarly to the proof of Lemma~\ref{lem:bc1}, we divide $R$ into $a_n+1$ many layers perpendicular to $e_n$. By Lemma~\ref{lem:bc1}, there is a proper edge $(2n-1)$-coloring of each layer of $R$ satisfying the  boundary condition for this layer. As in the proof of Lemma~\ref{lem:bc1}, we fix one such edge coloring and use it to color all layers of $R$ identically. We also color any edge that is adjacent to $R$ and parallel to $e_n$ by the color $c_n$, as required by the boundary condition. The remaining uncolored edges are now decomposed into $(a_1+1)\dots (a_{n-1}+1)$ many paths parallel to $e_n$. For each such path $p$ there is a color $c^*\in \{c_1,\dots, c_n,1,\dots, n\}$ such that $c^*\not\in C_x$ for all vertices $x$ in $p$, where $C_x$ is defined similarly as in the proof of Lemma~\ref{lem:bc1}. Note, however, that now the path $p$ has odd length. Thus we can use the two colors $c^*$ and $c_n$ in alternation to color all the edges of $p$, as done in the $n=1$ case. 
\end{proof}

Let $R$ be an $n$-dimensional rectangle of size $a_1\times\cdots\times a_n$, where $a_i$ is a positive even number for each $1\leq i\leq n$. Suppose
$$ R=[b_1,b_1+a_1]\times \cdots\times[b_n,b_n+a_n]$$
for $(b_1,\dots, b_n)\in\Z^n$. Let
$$ K=\left[b_1+\frac{a_1}{2}-1, b_1+\frac{a_1}{2}+1\right]\times \cdots\times\left[b_n+\frac{a_n}{2}-1, b_n+\frac{a_n}{2}+1\right]. $$
Then $K$ is a subrectangle of $R$ of size $2\times\cdots \times 2$, 
 and we call $K$ the {\em core} of $R$. An edge coloring of $C(\mathbb{Z}^n)$ is said to satisfy the {\em core condition} for $R$ if 
\begin{enumerate}
\item it is a proper edge $(2n+1)$-coloring of $R$ and its adjacent edges;
\item the color $n+1$ is only used in coloring some edges of the core of $R$.
\end{enumerate}

\begin{lemma}\label{middle}
Let $k$ be a nonnegative integer and let $d=4k+2$. For any $n$-dimensional rectangle $R$ of size $d\times \cdots\times d$, there is a proper edge $(2n+1)$-coloring satisfying both the boundary condition and the core condition for $R$. 
\end{lemma}
\begin{proof} If $k=0$, the core of $R$ is the same as $R$ and the lemma follows immediately from Lemma~\ref{lem:bc1}. For the rest of the proof we assume $k>0$. Again we only define an edge coloring for edges either in $R$ or adjacent to $R$.

Our edge coloring will be defined in $n+1$ many stages. Let $R_0=R$. In stage $1$, we first decompose $R_0$ into three parts $P_1, Q_1, R_1$ as follows. The vertices in $P_1$ have their coordinates in $R_0$ from the set
$$ [0, 2k-1]\times [0,d]\times \cdots \times[0,d];$$
the vertices in $R_1$ have their coordinates in $R_0$ from the set
$$ [2k,2k+2]\times [0,d]\times\cdots\times[0,d];$$
the vertices in $Q_1$ have their coordinates in $R_0$ from the set
$$ [2k+3,d]\times[0,d]\times\cdots\times[0,d].$$
Figure \ref{even} is an illustration of the construction at this stage. 
    \begin{figure}[h]
        \centering
        \begin{tikzpicture}[scale=0.8]
            \draw[draw=gray] (7.5,3.5) to (6.5,3.5);
            \draw[draw=gray] (7.5,0.5) to (6.5,0.5);
            \draw[draw=gray] (-5,1) to (-4,2);
            \draw[draw=gray] (-5,-2) to (-4,-1);
            \node at (0.5,6.5) {$\Big\uparrow$};
            \node at (1,6.5) {$e_1$};
		\draw[draw=gray] (5,5) to (7.5,7.5);
            \draw[draw=gray] (5,5) to (-5,5);
            \draw[draw=gray] (5,5) to (5,2);
            \draw[draw=gray] (5,1) to (5,-1);
            \draw[draw=gray] (5,-2) to (5,-5);
            \draw[draw=gray] (7.5,7.5) to (-2.5,7.5) to (-5,5) to (-5,2);
            \draw[draw=gray] (-5,1) to (-5,-1);
            \draw[draw=gray] (-5,-2) to (-5,-5) to (5,-5) to (7.5,-2.5) to (7.5,0.5);
            \draw[draw=gray] (7.5,1.5) to (7.5,3.5);
            \draw[draw=gray] (7.5,4.5) to (7.5,7.5);
            \draw[draw=gray] (-5,2) to (5,2) to (7.5,4.5);
            \draw[draw=gray] (-5,1) to (5,1) to (7.5,3.5);
            \draw[draw=gray] (-5,-1) to (5,-1) to (7.5,1.5);
            \draw[draw=gray] (-5,-2) to (5,-2) to (7.5,0.5);
            \node at (0,4) {$(2k-1)\times d\times\cdots\times d$};
            \node at (0,3) {edge $2n$-coloring};
            \node at (0,-3) {$(2k-1)\times d\times\cdots\times d$};
            \node at (0,-4) {edge $2n$-coloring};
            \node at (0,1.5) {$c_1$};
            \node at (-2.5,1.5) {$c_1$};
            \node at (2.5,1.5) {$c_1$};
            \node at (0,-1.5) {$c_1$};
            \node at (-2.5,-1.5) {$c_1$};
            \node at (2.5,-1.5) {$c_1$};
            \node at (0,0) {$2\times d\times\cdots\times d$};

\node at (-4,0) {$R_1$};
\node at (-4,-3.5) {$P_1$};
\node at (-4,3.5) {$Q_1$};

        \end{tikzpicture}
        \caption{The first stage of the definition of an edge coloring.}
        \label{even}
    \end{figure}
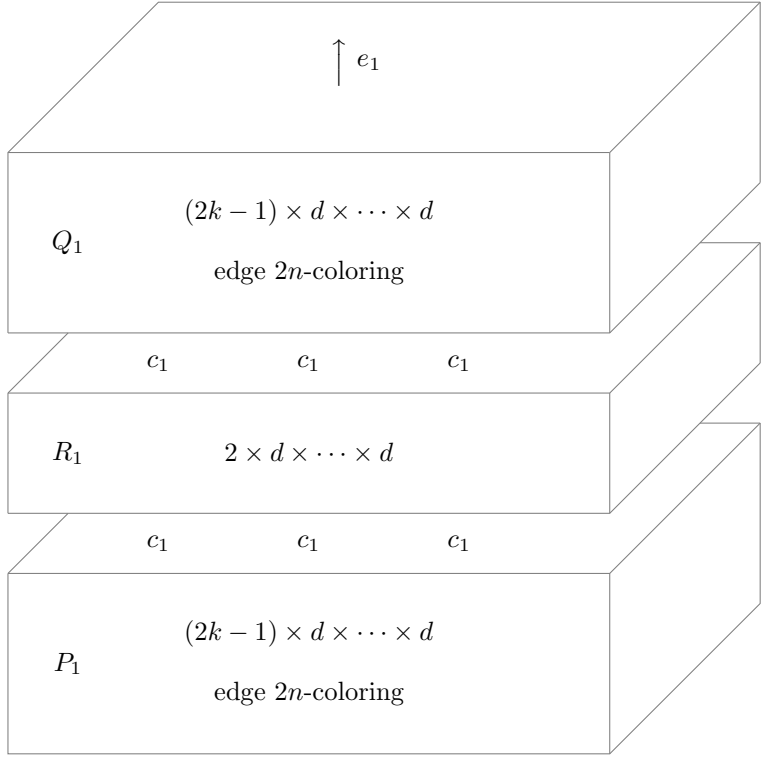

Now each of $P_1$ and $Q_1$ is a subrectangle of $R_0$ of size $(2k-1)\times d\times\cdots \times d$, and $R_1$ is a subrectangle of $R_0$ of size $2\times d\times\cdots\times d$. Since $2k-1$ is odd, we can apply Lemma~\ref{lem:bc2} to obtain a proper edge $2n$-coloring for $P_1$ and for $Q_1$ satisfying their boundary conditions. This is the end of stage 1. Note that the colors used so far are among $c_1,\dots, c_n, 1,\dots, n$, and the remaining uncolored edges are those of $R_1$. 

In stage 2, we repeat this construction by first decomposing $R_1$ into $P_2, Q_2, R_2$ according to the second coordinate of the vertices in $R_1$ and then properly edge $2n$-coloring the edges either in $P_2$ or $Q_2$ or adjacent to either of them. At the end of this stage, the remaining uncolored edges are those of $R_2$ and the colors used so far are still among $c_1,\dots, c_n, 1, \dots, n$.

Repeating this construction $n$ times according to each of the $i$-th coordinates, we obtain at the end of stage $n$ a subrectangle $R_n$ of size $2\times\cdots\times 2$. All those edges adjacent to $R_n$ have been colored to satisfy the boundary condition for $R_n$. Applying Lemma~\ref{lem:bc1}, we get a proper edge $(2n+1)$-coloring of $R_n$ which satisfies the boundary condition for $R_n$. This resulting edge coloring is proper and obviously satisfies the core condition for $R$.
\end{proof}

Next we prove a generalization of Lemma~\ref{middle} in which the core condition is replaced by a shifted core condition, as follows. Let $R$ be an $n$-dimensional rectangle of size $a_1\times\cdots\times a_n$, where $a_i$ is a postive even number for each $1\leq i\leq n$. Suppose
$$ R=[b_1,b_1+a_1]\times \cdots\times[b_n,b_n+a_n]$$
for $(b_1,\dots, b_n)\in\Z^n$. Let $K$ be the core of $R$. For $t=(t_1,\dots, t_n)\in \Z^n$ where 
$$ -\frac{a_i}{2}+1\leq t_i\leq \frac{a_i}{2}-1 $$
for each $1\leq i\leq n$, let
$$K+t=\{(x_1+t_1,\dots,x_n+t_n)\colon (x_1,\dots, x_n)\in K\}. $$
Then $K+t$ is a subrectangle of $R$ of size $2\times\cdots\times 2$. We call $K+t$ the {\em $t$-shifted core} of $R$. An edge coloring of $C(\mathbb{Z}^n)$ is said to satisfy the {\em $t$-shifted core condition} for $R$ if 
\begin{enumerate}
\item it is a proper edge $(2n+1)$-coloring of $R$ and its adjacent edges;
\item the color $n+1$ is only used in coloring some edges of the $t$-shifted core of $R$.
\end{enumerate}

\begin{lemma}\label{lem:shiftedcore}
Let $k$ be a non-negative integer and let $d=4k+2$. For any $n$-dimensional rectangle $R$ of size $d\times \cdots\times d$ and for any $t=(t_1,\dots, t_n)\in [-2k+2,2k-2]^n$ where $t_i$ is an even number for each $1\leq i\leq n$, there is a proper edge $(2n+1)$-coloring satisfying both the boundary condition and the $t$-shifted core condition for $R$. 
\end{lemma}

\begin{proof} The proof is identical to that of Lemma~\ref{middle}, except that at each stage of the construction $P_i$ will have size $2\times \cdots\times 2\times(2k-1+t_i)\times d\times\cdots\times d$ and $Q_i$ will have size $2\times\cdots \times 2\times (2k-1-t_i)\times d\times\cdots\times d$. Since $2k-1+t_i$ and $2k-1-t_i$ are both odd, Lemma~\ref{lem:bc2} can be applied to complete the proof.
\end{proof}

In our proof of the main theorem below, we view an application of Lemma~\ref{lem:shiftedcore} as a two-step process. In the first step, we apply Lemma~\ref{middle} to obtain a proper edge $(2n+1)$-coloring of $R$ satisfying the core condition. Then, in the second step, we may shift the core to any other position within $R$ as long as each coordinate of the shift vector used is even. This is not quite what happens in the proof of Lemma~\ref{lem:shiftedcore}, but it is a useful point of view to describe the algorithm in the proof of our main theorem.

\subsection{Continuous proper edge colorings} In this final subsection we prove the main theorem of the paper. We will use the following lemmas from \cite{GJ15}.

\begin{lemma}\label{MarkerPoint}
    (Basic clopen marker lemma \cite[Lemma 2.1]{GJ15}) Let $\rho$ be the path distance on the Schreier graph $F(2^{\mathbb Z^n})$. Then for any positive integer $d$, there is a relatively clopen set $M_d\subseteq F(2^{\mathbb Z^n})$ such that
    \begin{enumerate}
        \item if $x,y\in M_d$ are distinct then $\rho(x,y)>d$;
        \item for any $x\in F(2^{\mathbb Z^n})$ there is $y\in M_d$ such that $\rho(x,y)\leq d$.
    \end{enumerate}
\end{lemma}
The integer $d$ in the above lemma is usually called the {\em marker distance}, the set $M_d$ is called a {\em marker set}, and its elements are called {\em marker points}. The careful reader might note that the path distance $\rho$ we use in this paper is the $\ell_\infty$ distance on $\mathbb{Z}^n$, while in \cite{GJ15} the distance $\rho$ was the $\ell_1$ distance. Since the proof for Lemma~\ref{MarkerPoint} only depends on the finiteness of the balls in $\mathbb{Z}^n$ under these distances, the lemma continues to hold for the path distance we use here. However, note that the balls under the $\ell_\infty$ distance are $n$-dimensional rectangles, and hence are more convenient to use in our discussions below.

\begin{lemma}\label{MarkerRegion}
    (Marker regions lemma \cite[Theorem 3.1]{GJ15}) Let $d>0$ be an integer. Then there is a relatively clopen subequivalence relation $R^n_d$ on $F(2^{\mathbb{Z}^n})$ such that each of the $R^n_d$-equivalence classes is of the form $R\cdot x$, where $R$ is an $n$-dimensional rectangle with side lengths either $d$ or $d+1$. Here $R^n_d$ is relatively clopen means that $\left\{(x,g)\in F(2^{\mathbb Z^n})\times\mathbb Z^n\colon (x, g\cdot x)\in R^n_d\right\}$ is a clopen subset of $F(2^{\mathbb Z^n})\times\mathbb Z^n$.
\end{lemma}

The $R^n_d$-equivalence classes in the above lemma are called {\em marker regions}.

All the terminology we defined in the preceding subsection about an $n$-dimensional rectangle $R$ in $\mathbb{Z}^n$ can be similarly defined for a marker region of the form $R\cdot x$ for $x\in F(2^{\mathbb{Z}^n})$. Thus we may refer to a marker region as an $n$-dimensional rectangle in the Schreier graph $F(2^{\Z^n})$. In particular, for a continuous edge coloring of $F(2^{\mathbb{Z}^n})$, we may consider whether it satisfies the boundary condition and the (shifted) core condition for a marker region, and when we write the marker region as $R\cdot x$, the boundary condition and the (shifted) core condition are independent from the choice of $R$ and $x$. Note that the (shifted) core condition is only valid when the side lengths of the marker region are even.

\begin{proposition}\label{standard}
    Let $k$ be a non-negative integer and let $d=4k+2$. Then there is a relatively clopen subequivalence relation $R^n_d$ on $F(2^{\mathbb{Z}^n})$ and a continuous proper edge $(2n+1)$-coloring of $F(2^{\mathbb{Z}^n})$ such that 
\begin{enumerate}
\item[(i)] each of the $R^n_d$-equivalence classes is of the form $R\cdot x$, where $R$ is an $n$-dimensional rectangle with side lengths either $d$ or $d+1$;
\item[(ii)] the edge coloring satisfies both the boundary condition and the core condition for each of the $R^n_d$-equivalence classes.
\end{enumerate}
\end{proposition}

\begin{proof} We apply Lemma~\ref{MarkerRegion} to obtain a relatively clopen subequivalence relation $R^n_d$ on $F(2^{\mathbb{Z}^n})$ so that each of the $R^n_d$-equivalence classes is of the form $R\cdot x$, where $R$ is an $n$-dimensional rectangle with side lengths either $d$ or $d+1$. Consider a fixed such $R$ and assume it is of size $a_1\times\cdots a_n$, where $a_i\in\{d, d+1\}$ for each $1\leq i\leq n$. If any $a_i$ is $d+1$, we obtain and fix a proper edge $2n$-coloring $f_R$ satisfying the boundary condition for $R$ by Lemma~\ref{lem:bc2}. Otherwise, $R$ is of size $d\times\cdots\times d$, and we obtain and fix a proper edge $(2n+1)$-coloring $f_R$ satisfying both the boundary condition and the core condition for $R$ by Lemma~\ref{middle}. Note that $f_R$ only depends on the sizes of $R$. 

Now we define a proper edge $(2n+1)$-coloring $c$ of $F(2^{\mathbb{Z}^n})$ as follows. For a standard generator $s=\pm e_i$ and $x\in F(2^{\mathbb{Z}^n})$, let 
$$ c(x,s\cdot x)=\left\{\begin{array}{ll}
f_R(\overline{0},s), & \mbox{ if $(x, s\cdot x)\in R^n_d$, $R$ is an $n$-dimensional rectangle,}\\
& \mbox{ and $R\cdot x$ is 
the $R^n_d$-class containing $x$;}\\
c_i, & \mbox{ if $(x,s\cdot x)\not\in R^n_d$.}
\end{array}\right.
$$
Then $c$ is proper because each $f_R$ satisfies the boundary condition, and $c$ is continuous because for an $n$-dimensional rectangle $R$, $R\cdot x$ is the $R^n_d$-class containing $x$ if and only if $\overline{0}\in R$.
\end{proof}

In particular we obtain a continuous proper edge $(2n+1)$-coloring of $F(2^{\Z^n})$. This generalizes  \cite[Theorem 3.1.4]{GJKS23} and answers  \cite[Question 10.3]{GS} for $n\geq 3$, which was also stated implicitly in \cite{GJKS23} (see the remark after  \cite[Question 3.1.5]{GJKS23}).

\begin{corollary}\label{cor:main} For any $n\geq 1$, $\chi'_c(F(2^{\Z^n}))=2n+1$.
\end{corollary}

Next, we generalize Corollary~\ref{cor:main} to the case of arbitrary generating sets. Again we assume tacitly that our generating sets are symmetric and do not contain the identity. Let $\Lambda$ be a subgroup of $\mathbb{Z}^n$ and let $S$ be a generating set of $\Lambda$, i.e., $\Lambda=\langle S\rangle$. We say that $S$ is {\em linearly semi-independent} if for any $s\in S$, $\langle s\rangle\cap \langle S\setminus\{s,-s\}\rangle=\{\bar{0}\}$. If $S=2m$ where $m\leq n$ and $S$ is linearly semi-independent, the Cayley graph $C(\Lambda, S)$ is isomorphic to the Cayley graph $C(\Z^m)$ with the standard generating set. We have the following generalizations of Lemmas~\ref{MarkerPoint} and \ref{MarkerRegion}.

\begin{lemma}\label{GenMarkerPoint}
   Let $\Lambda$ be a subgroup of $\mathbb Z^n$ and let $S$ be a generating set of $\Lambda$. Assume $S$ is linearly semi-independent. Let $\Lambda\acts F(2^{\mathbb Z^n})$ be the restriction of the Bernoulli shift action $\mathbb{Z}^n\curvearrowright F(2^{\mathbb{Z}^n})$ on $\Lambda\times F(2^{\mathbb Z^n})$. Let $\rho_S$ be the path distance on the Schreier graph $G(\Lambda,S,F(2^{\mathbb Z^n}))$. Then for any positive integer $d$, there is a relatively clopen set $M_d\subseteq F(2^{\mathbb Z^n})$ such that
    \begin{enumerate}
        \item if $x,y\in M_d$ are distinct then $\rho_S(x,y)>d$;
        \item for any $x\in F(2^{\mathbb Z^n})$ there is $y\in M_d$ such that $\rho_S(x,y)\leq d$.
    \end{enumerate}
\end{lemma}

\begin{lemma}\label{GenMarkerRegion}
Let $\Lambda$ be a subgroup of $\mathbb Z^n$ and let $S$ be a generating set of $\Lambda$. Assume $S$ is linearly semi-independent. Let $\Lambda\acts F(2^{\mathbb Z^n})$ be the restriction of the Bernoulli shift action $\mathbb{Z}^n\curvearrowright F(2^{\mathbb{Z}^n})$ on $\Lambda\times F(2^{\mathbb Z^n})$.  Let $d>0$ be an integer. Assume $|S|=2m$ where $m\leq n$. Then there is a relatively clopen subequivalence relation $R^S_d$ on $F(2^{\mathbb{Z}^n})$ such that each of the $R^S_d$-equivalence classes is an $m$-dimensional rectangle in the Schreier graph $G(\Lambda, S, F(2^{\Z^n}))$ with side lengths either $d$ or $d+1$. 
\end{lemma}

The proofs of these lemmas are identical to those of Lemmas~\ref{MarkerPoint} and \ref{MarkerRegion}. In fact, all of our discussions about the standard generating set so far in this paper apply to these more general linearly semi-independent generating sets. In particular we have the following generalization of Proposition~\ref{standard} with the same proof.

\begin{proposition}\label{nonstandard}
Let $\Lambda$ be a subgroup of $\mathbb Z^n$ and let $S$ be a generating set of $\Lambda$. Assume $S$ is linearly semi-independent. Let $\Lambda\acts F(2^{\mathbb Z^n})$ be the restriction of the Bernoulli shift action $\mathbb{Z}^n\curvearrowright F(2^{\mathbb{Z}^n})$ on $\Lambda\times F(2^{\mathbb Z^n})$. Let $k$ be a non-negative integer and let $d=4k+2$. Assume $|S|=2m$ where $m\leq n$. Then there is a relatively clopen subequivalence relation $R^S_d$ on $F(2^{\mathbb{Z}^n})$ and a continuous proper edge $(2m+1)$-coloring of $G(\Lambda, S, F(2^{\mathbb{Z}^n}))$ such that 
\begin{enumerate}
\item[(i)] each of the $R^S_d$-equivalence classes is an $m$-dimensional rectangle with side lengths either $d$ or $d+1$;
\item[(ii)] the edge coloring satisfies both the boundary condition and the core condition for each of the $R^S_d$-equivalence classes.
\end{enumerate}
\end{proposition}

We are now ready for the main theorem of this paper.

\begin{theorem}
    For any $n\geq 1$ and any generating set $S$ of $\mathbb {Z}^n$, let $G=F(S,2^{\mathbb{Z}^n})$. Then $\chi'_c(G)=\chi'(G)+1=|S|+1$.
\end{theorem}
\begin{proof} It is easy to see that $\chi'(G)=|S|$. The lower bound has been established by Corollary~\ref{cor:lb}. We only need to construct a continuous proper edge $(|S|+1)$-coloring of $F(S, 2^{\mathbb{Z}^n})$. 

Inductively define $S_i$ so that
\begin{enumerate}
\item[(i)] $S_0$ is a maximal linearly semi-independent subset of $S$.
\item[(ii)] If $S_{i-1}$ is defined for $i\geq 1$, let $S_i$ be a maximal linearly semi-independent subset of $S\setminus\bigcup_{j=0}^{i-1}S_j$.
\end{enumerate}
Then there is some $m\geq 0$ such that $S=\bigcup_{i=0}^m S_i$, where each $S_i$ is symmetric, and in particular $|S_0|=2n$. Assume for $1\leq i\leq m$, $|S_i|=2n_i$. By the maximal linear semi-independence, for each $1\leq i\leq m$, let $k_i>0$ be the smallest natural number such that the scalar multiplication $k_i\langle S_i\rangle$ is contained in $\langle S_{i-1}\rangle$.

Let $\alpha=3^{n}m$ and $\beta=2\cdot 3\cdot \prod_{i=1}^mk_i$. Fix an arbitrary non-identity $s=(s_1,\dots, s_n)\in S_m$ and let $|s|=\sum_{k=1}^n|s_k|$.  For each $1\leq i\leq m$, since $\beta$ is a multiple of $\prod_{j=i+1}^mk_j$, we have $\beta s\in \bigcap_{j=i}^m \langle S_j\rangle$; in particular $\beta s\in\langle S_i\rangle$. Thus there exist $a_{i,1},\dots, a_{i,n_i}\in \mathbb{Z}$ such that $\beta s=a_{i,1}e_{i,1}+\cdots+a_{i,n_i}e_{i,n_i}$, where $S_i=\{\pm e_{i,1},\dots, \pm e_{i,n_i}\}$. Let $\gamma=\sup\{|a_{i,j}|\colon 1\leq i\leq m, 1\leq j\leq n_i\}$. Let $d=4\cdot 2(\gamma+1)(\alpha+1)(\beta+1)|s|+2$. 

Consider the Schreier graph $G_i=G(\langle S_i\rangle, S_i, F(2^{\Z^n}))$. Let $\rho_i=\rho_{S_i}$ be the path distance on $G_i$. We have $E(G)=\bigcup_{i=0}^m E(G_i)$. For each $0\leq i\leq m$, let $H_i$ be the Schreier graph on $F(2^{\Z^n})$ where $E(H_i)=\bigcup_{j=i}^m E(G_j)$, and let $C_i$ be a set of $2n_i$ many fresh colors, i.e., for $i\neq j$, $C_i\cap C_j=\varnothing$. Note that $H_i$ is generated by the actions of $\bigcup_{j=i}^m S_j$. Assume $0$ is a color distinct from any color in $\bigcup_{i=0}^m C_i$. Let $N_i=1+\sum_{j=i}^m 2n_j$. To define a continuous edge $(|S|+1)$-coloring on $G$, we use reverse induction on $i=m,\dots, 0$ to define a continuous edge $N_i$-coloring of $H_i$ with the colors in $\{0\}\cup\bigcup_{j=i}^m C_j$. When $i=0$ the resulting edge coloring is as required.

When $i=m$ we apply Proposition~\ref{nonstandard} to $S_m$ to obtain a relatively clopen subequivalence relation $R^m_d=R^{S_m}_d$ on $F(2^{\Z^n})$ and a continuous proper edge $N_m$-coloring of $H_m=G_m$. Let $\kappa_m$ denote this coloring. Then $\kappa_m$ uses colors in $\{0\}\cup C_m$. For each $R^m_d$-marker region, which is an $n_m$-dimensional rectangle with side lengths either $d$ or $d+1$, $\kappa_m$ satisfies both the boundary condition and the core condition. In particular, the boundary condition guarantees that $\kappa_m$ is a proper edge coloring of $G_m$. Here the core condition means that the color $0$ is only used in coloring some edges of the cores of $R^m_d$-marker regions. Let $K_m$ be the union of all cores of the $R^m_d$-marker regions. Then $K_m$ is a clopen subset of $F(2^{\Z^n})$. 

For the clarity of our proof, consider $i=m-1$ next. Now we apply Proposition~\ref{nonstandard} to $S_{m-1}$ to obtain a relatively clopen subequivalence relation $R^{m-1}_d=R^{S_{m-1}}_d$ on $F(2^{\Z^n})$ and a continuous proper edge $(2n_{m-1}+1)$-coloring of $G_{m-1}$. Let $\eta_{m-1}$ denote this coloring. Then $\eta_{m-1}$ uses colors in $\{0\}\cup C_{m-1}$. For each $R^{m-1}_d$-marker region, which is an $n_{m-1}$-dimensional rectangle with side lengths $d$ or $d+1$, $\eta_{m-1}$ satisfies both the boundary condition and the core condition. Here again the core condition means that the color $0$ is only used in coloring some edges of the cores of $R^{m-1}_d$-marker regions. Let $Q_{m-1}$ be the union of all cores of the $R^{m-1}_d$-marker regions. Then $Q_{m-1}$ is a clopen subset of $F(2^{\Z^n})$. 

Consider a particular $R^{m-1}_d$-marker region $R$ and its core $K$. We claim that one of the following $\alpha+1$ many sets
$$ K,\ \beta s\cdot K,\ 2\beta s\cdot K,\ \dots, \ \alpha\beta s\cdot K $$
has empty intersection with $K_m$. To see this, note that $\beta$ is a multiple of $k_m$, and hence $\beta s\in \langle S_{m-1}\rangle$. Since $K$ is an $n_{m-1}$-dimensional rectangle in $G_{m-1}$ with side lengths $2$, $\beta$ is a multiple of $3$ and $\frac{\beta}{3}s\in \langle S_{m-1}\rangle$, we have that the above $\alpha+1$ many sets are pairwise disjoint. Since $d>4(\alpha+1)\gamma\beta|s|$, all of the above $\alpha+1$ many sets are still contained in $R$. On the other hand, for each point $x\in K$, consider 
$$ x,\ \beta s\cdot x,\ 2\beta s\cdot x,\ \dots,\ \alpha\beta s\cdot x. $$
If any of these $\alpha+1$ many points is in $K_m$, say in the core of a particular $R^m_d$-marker region $R'$, then all of the displayed $\alpha+1$ many points are still in $R'$, and none of the other $\alpha$ many points are in $K_m$ anymore. This means that at most one of the displayed $\alpha+1$ many points can be in $K_m$. 
For every $x\in K$, let $I_x=\{i\beta s\cdot x\colon i\leq \alpha\}$. Now, since $|K|=3^{n_{m-1}}\leq 3^n\leq \alpha$, we have that if all these $\alpha+1$ many shifts of $K$ intersect $K_m$, as $\bigsqcup_{i\leq \alpha}i\beta s\cdot K=\bigsqcup_{x\in K}I_x$, at leat one $I_x$ would contain two points in $K_m$, a contradiction.

Now we define the coloring $\kappa_{m-1}$, which is an extension of $\kappa_m$. For each $R^{m-1}_d$-marker region $R$ and its core $K\subseteq Q_{m-1}$, we find the least integer $a$ such that $0\leq a\leq \alpha$ and $a\beta s\cdot K\cap K_m=\varnothing$. If $R$ has an side of length $d+1$ then we let $\kappa_{m-1}$ on $R$ and its adjacent edges to be $\eta_{m-1}$ on $R$ and its adjacent edges. Note that  only colors in $C_{m-1}$ are used in $\eta_{m-1}$. Otherwise, apply Lemma~\ref{lem:shiftedcore} to obtain a proper edge $(2n_{m-1}+1)$-coloring $\kappa_{m-1}$ on $R$ satisfying both the boundary condition and the $a\beta s$-shifted core condition for $R$. Again the shifted core condition means that the color $0$ is only used in coloring some edges of $a\beta s\cdot K$. This is possible since $d=4k+2$ for some $k$, $a\beta$ is even and $\frac{a\beta}{2}s\in\langle S_{m-1}\rangle$. The shifted core condition guarantees that no adjacent edges are both colored $0$ in the resulting $\kappa_{m-1}$. Furthermore, no adjacent edges are colored the same color in $\kappa_{m-1}$ since $C_{m-1}\cap C_m=\varnothing$. Together with this observation, the boundary condition in the definition guarantees that $\kappa_{m-1}$ is a proper edge $N_{m-1}$-coloring of $H_{m-1}$. It is clear from the definition that $\kappa_{m-1}$ is continuous. This finishes the definition of $\kappa_i$ as a continuous proper edge $N_i$-coloring of $H_i$ for $i=m-1$. To set up the general induction, we let $K_{m-1}$ be the union of all shifted cores of the $R^{m-1}_d$-marker regions as defined above. More specifically,
$$\begin{array}{rl} K_{m-1}=\bigcup\{a\beta s\cdot K\colon\!\!\!\! & \mbox{$K$ is a core of some $R^{m-1}_d$-marker region} \\
& \mbox{ and $0\leq a\leq \alpha$ is the least such that $a\beta s\cdot K\cap K_m=\varnothing$}\}.
\end{array} $$
Then $K_{m-1}$ is still a clopen subset of $F(2^{\Z^n})$. Note that the original cores in $Q_{m-1}$ can be shifted to at most $\gamma\alpha\beta|s|$ away to form $K_{m-1}$; however, since $d>4\cdot 2\gamma\alpha\beta|s|$, each shifted core in $K_{m-1}$ is still at least $\gamma\alpha\beta|s|$ away from the boundaries of the $R^{m-1}_d$-marker regions.

For the general $i<m-1$, assume inductively that we have defined a continuous proper edge $N_{i+1}$-coloring $\kappa_{i+1}$ of $H_{i+1}$. Also assume that for all $i+1\leq j\leq m$ we have defined relatively clopen subequivalence relations $R^j_d$ of $F(2^{\Z^n})$ and relatively clopen sets $K_j$ as shifted cores of $R^j_d$-marker regions. 

Apply Proposition~\ref{nonstandard} to $S_i$ to obtain a relatively clopen subequivalence relation $R^i_d=R^{S_i}_d$ on $F(2^{\Z^n})$ and a continuous proper edge $(2n_i+1)$-coloring of $G_i$. Let $\eta_i$ denote this coloring. Then $\eta_i$ uses colors in $\{0\}\cup C_i$. For each $R^i_d$-marker region, which is an $n_i$-dimensional rectangle with side lengths $d$ or $d+1$, $\eta_i$ satisfies both the boundary condition and the core condition. For a particular $R^i_d$-marker region $R$ and its core $K$, we claim that at least one of the $\alpha+1$ many sets 
$$ K,\ \beta s\cdot K,\ 2\beta s\cdot K,\ \dots, \ \alpha\beta s\cdot K $$
has empty intersection with $\bigcup_{j=i+1}^mK_j$. The argument for this claim is similar to the above proof, except that we now use the fact $\beta s\in\langle S_i\rangle$ to see that the above $\alpha+1$ many displayed sets are still in $R$. For any $x\in K$ consider the points
$$ x,\ \beta s\cdot x,\ 2\beta s\cdot x,\ \dots,\ \alpha\beta s\cdot x. $$
For any $j$ with $i+1\leq j\leq m$, if any of these points is in $K_j$, which is in some $R^j_d$-marker region $R'$, then by induction, the entire sequence is still at least $\gamma\alpha\beta|s|$ away from the boundary of $R'$, and therefore still contained in $R'$. It follows that at most one of these points can be in $K_j$, and thus at most $m$ many of these points can be in $\bigcup_{j=i+1}^mK_j$. Now since $|K|=3^{n_i}\leq 3^n$, there are at most $3^nm=\alpha$ many points in the displayed $\alpha+1$ many shifts of $K$ which can be in $\bigcup_{j=i+1}^m K_j$. It follows that there are at most $\alpha$ many of the displayed sets which can have nonempty intersection with $\bigcup_{j=i+1}^m K_j$. The claim again follows from the pigeonhole principle. 

We define the coloring $\kappa_i$ as an extension of $\kappa_{i+1}$. For each $R^i_d$-marker region $R$ and its core $K$, we find the least integer $a$ such that $0\leq a\leq \alpha$ and $a\beta s\cdot K\cap \bigcup_{j=i+1}^m K_j=\varnothing$. If $R$ has an side of length $d+1$ then we let $\kappa_i$ on $R$ and its adjacent edges to be $\eta_i$. Note that in this case only colors in $C_i$ are used. Otherwise, apply Lemma~\ref{lem:shiftedcore} to obtain a proper edge $(2n_i+1)$-coloring $\kappa_i$ on $R$ satisfying both the boundary condition and the $a\beta s$-shifted core condition for $R$. This is possible since $d=4k+2$ for some $k$ and $a\beta$ is even. The shifted core condition guarantees that no adjacent edges are both colored $0$ in the resulting $\kappa_i$. Furthermore, no adjacent edges are colored the same color in $\kappa_i$ since $C_i\cap C_j=\varnothing$ for all $i+1\leq j\leq m$. Together with this observation, the boundary condition in the definition guarantees that $\kappa_i$ is a proper edge $N_i$-coloring of $H_i$. It is clear from the definition that $\kappa_i$ is continuous. This finishes the definition of $\kappa_i$ as a continuous proper edge $N_i$-coloring of $H_i$. To continue the induction, we let $K_i$ be the set of all shifted cores of the $R^i_d$-marker regions as defined above. Then $K_i$ is still a clopen subset of $F(2^{\Z^n})$.

To complete the proof, note that $\kappa_0$ is a continuous proper edge $N_0$-coloring of $H_0$, where $H_0=F(S, 2^{\Z^n})$ and $N_0=|S|+1$.
\end{proof}

\end{document}